\documentclass[11pt]{amsart}

\usepackage{amsmath, amssymb, amsbsy, amscd, mathtools}
\usepackage{times, mathrsfs, euscript, thmtools}
\usepackage{epsf, overpic, psfrag, stmaryrd}
\usepackage{caption, subcaption}
\usepackage[all]{xy}
\usepackage{hyperref}
\usepackage[dvipsnames]{xcolor}
\usepackage{enumitem, bbm}
\usepackage{tikz, tikz-cd}
\usetikzlibrary{decorations.markings}
\tikzcdset{scale cd/.style={
    every label/.append style={scale=#1},
    cells={nodes={scale=#1}}}}

\usepackage{array}
\usepackage{todonotes}
\usepackage{tikz}

\usepackage{caption}
\usepackage{subcaption}
\newtheorem{theorem}{Theorem}[section]

\newtheorem{conjecture}[theorem]{Conjecture}
\newtheorem{corollary}[theorem]{Corollary}
\newtheorem{lemma}[theorem]{Lemma}
\newtheorem{proposition}[theorem]{Proposition}

\theoremstyle{definition}
\newtheorem{definition}[theorem]{Definition}
\newtheorem{example}[theorem]{Example}

\newtheorem{remark}[theorem]{Remark}

\numberwithin{equation}{subsection}

\DeclareMathAlphabet{\mathpgoth}{OT1}{pgoth}{m}{n}

\DeclareMathAlphabet{\mathpzc}{OT1}{pzc}{m}{it}

\newcommand{\be}{\begin{enumerate}}
\newcommand{\ee}{\end{enumerate}}
\newcommand{\op}{\operatorname}

\newcommand{\tp}{\mathsf}
\newcommand{\X}{\mathbb X}

\newcommand{\R}{\mathbb{R}}

\DeclareMathOperator{\ind}{Ind}

\DeclareMathOperator{\crit}{Crit}
\newcommand{\hess}[1]{\text{Hess}_{#1}}

\DeclareMathOperator{\stab}{stab}

\newcommand{\D}[1]{\mathcal{D}_{#1}}

\newcommand{\grad}{\op{grad}}

\newcommand{\A}[1]{\mathcal{A}_{#1}}

\newcommand{\M}[2]{\mathcal{M}(#1,#2)}

\newcommand{\dd}{\mathrm{d}}

\newcommand{\p}{\partial}

\newcommand{\Z}{\mathbb{Z}}

\newcommand{\Q}{\mathbb{Q}}

\newcommand{\coC}{C^{\op{co}}}

\newcommand{\coH}{H^{\op{co}}}

\newcommand{\cop}{\p^{\op{co}}}

\newcommand{\inC}{C^{\op{in}}}

\newcommand{\inp}{\p^{\op{in}}}

\newcommand{\inH}{H^{\op{in}}}

\makeatletter
\newcommand{\labitem}[2]{%
  \def\@itemlabel{#1}
  \item
  \def\@currentlabel{#1}\label{#2}}
\makeatother

\begin{document}

\title[Invariant and coinvariant Morse homologies for orbifolds]{Invariant and coinvariant Morse homologies for orbifolds}

\author{Erkao Bao}
\address{School of Mathematics, University of Minnesota, Minneapolis, MN 55455}
\email{bao@umn.edu}
\urladdr{https://erkaobao.github.io/math/}


\author{Lina Liu}
\address{School of Mathematics, University of Minnesota, Minneapolis, MN 55455}
\email{liu02226@umn.edu}
\urladdr{https://sites.google.com/view/linaliu/}


\begin{abstract}
  In this note, we construct invariant and coinvariant Morse chain complexes with integer coefficients for any compact effective orbifold. We show that the homologies of these two chain complexes are invariants of the orbifold. We conjecture that the homology of the coinvariant chain complex computes the singular homology of the underlying topological space with $\Z$-coefficients, thereby refining the construction in \cite{cho2014orbifold}, which recovers the homology over $\Q$. In contrast, the homology of the invariant Morse chain complex is sensitive to the orbifold structure.
\end{abstract}

\maketitle

\setcounter{tocdepth}{2}
\section{Introduction}

Morse homology on orbifolds is an interesting subject in its own right and also serves as a prototype for various Floer-type theories, such as cylindrical contact homology. However, a complete and satisfactory formulation of Morse homology in the orbifold setting remains underdeveloped, mainly due to the following major challenges:
\begin{enumerate}[label=(\alph*)]
  \item \label{challenge a} \emph{Equivariant transversality}: given a Morse function $f$ on an orbifold, there does not, in general, exist a metric $g$ such that the pair $(f,g)$ is Morse--Smale;
  \item \label{challenge b} Even when the Morse--Smale condition holds, the moduli spaces of Morse flow lines are typically orbifolds rather than manifolds;
  \item \label{challenge c} Some group actions reverse the orientations of descending manifolds.
\end{enumerate}

To address \ref{challenge a}, one can employ the multisection technique by introducing multi-valued metrics or multi-vector fields.  
This idea parallels the use of multisections in the construction of Kuranishi structures around the moduli spaces of pseudoholomorphic curves, as in \cite{fukaya1999arnold}.  
However, this approach typically leads to a theory defined only over~$\Q$.

The issue in \ref{challenge b} can be resolved by working with weighted branched manifolds, which again usually forces the theory to be defined over~$\Q$.

Challenge \ref{challenge c} arises when the orbifold is not a global quotient.  
In \cite{cho2014orbifold}, the authors propose discarding the non-orientable critical points. While this circumvents the orientation problem, it fails to recover the homology of the underlying topological space with integer coefficients. The following example illustrates this limitation.

\begin{example} \label{example: RPn}
  Let $S^{n} \subset \R^{n+1}$, and let $G = \langle \sigma \rangle$ act on $S^n$ via 
  \[
    \sigma(x_0, x_1, \dots, x_n) = (x_0, -x_1, \dots, -x_n).
  \]
  Consider the global quotient orbifold $[S^n/G]$. The underlying topological space is homeomorphic to $S\mathbb{RP}^{n-1}$, the suspension of $\mathbb{RP}^{n-1}$, whose homology satisfies $H_k(S\mathbb{RP}^{n-1}) \simeq \Z_2$ for all odd $1 \le k < n$.  
  Let $f: S^n \to \R$ be the Morse function $f(x_0, x_1, \dots, x_n) = x_0$. The induced Morse function on the quotient orbifold $[S^n/G]$ has only two critical points, $(1,0,\dots,0)$ and $(-1,0,\dots,0)$. For $n$ large, these two critical points cannot generate the full homology.
\end{example}

There are several related works addressing various aspects related challenges, such as \cite{seidel2010localization,hendricks2016flexible,bao2021equivariant,abouzaid2021arnold,bai2025newtransversalityconditionorbifolds}, each operating under different assumptions and frameworks.

In this note, we adopt the \emph{stabilization approach} developed in \cite{bao2024morse}. Roughly speaking, a Morse critical point $p$ is said to be \emph{stable} if its descending manifold is contained in the fixed locus of the stabilizer of $p$. In particular, the critical point $(1,0,\dots,0)$ in Example~\ref{example: RPn} is not stable.
Given an arbitrary Morse function on an orbifold, we locally modify it near non-stable critical points to introduce a few additional stable ones, thereby stabilizing the original critical points. This modification can be carried out explicitly, and by letting the size of the perturbation tend to zero, one can compare the moduli spaces of Morse flow lines for the modified and original Morse functions (see Section 2.4 in \cite{kronheimer2007Monopoles} and \cite{bao2024equivariantmorsehomologyreflection} for special cases).

When all critical points are stable, we show that challenges \ref{challenge a}, \ref{challenge b}, and \ref{challenge c} all disappear.

We then construct two Morse chain complexes for orbifolds: the \emph{invariant complex} and the \emph{coinvariant complex}. The construction is motivated by the global quotient case. Suppose our orbifold is of the form $[M/G]$. The stabilization process yields a Morse chain complex $C$ equipped with a $G$-action. From this data, one obtains:
\begin{itemize}
  \item the \emph{invariant complex} $\inC$, the subcomplex of $C$ fixed under the $G$-action;
  \item the \emph{coinvariant complex} $\coC$, obtained by quotienting $C$ by the subcomplex generated by elements of the form $c - gc$ for all $c \in C$ and $g \in G$.
\end{itemize}
The homology of the coinvariant complex computes the homology of the topological space $M/G$, while the homology of the invariant complex reflects aspects of the orbifold structure.

For general orbifolds that are not global quotients, the analogue of $C$ is not immediately apparent. Nevertheless, both the invariant and coinvariant chain complexes can still be defined.  
In this paper, we provide precise definitions of these complexes, establish their independence from the choices of Morse function and Riemannian metric, and compute several illustrative examples. 
We conjecture that the homology of the coinvariant complex computes the singular homology of the underlying topological space with integer coefficients. In contrast, the homology of the invariant complex is sensitive to the orbifold structure and is expected to be useful to detect whether an orbifold is a global quotient by a finite group action (see Example~\ref{example: sphere}).

\textbf{Acknowledgments} The authors thank Tyler Lawson for insightful discussions. The first author thank Cheol-Hyun Cho, Hansol Hong and Guangbo Xu for helpful conversations. During part of the project, the second author was hosted by Michael Hutchings and Nancy Mae Eagles at UC Berkeley. 
The first author was supported by NSF Grant DMS-2404529,
and the second author was supported by NSF GRFP Grant 2237827.

\section{Morse functions on orbifolds}\label{Section: morse fun on orbifolds}

Let $\X$ be an effective orbifold, and let $\tp X$ denote its underlying topological space.  
We refer to \cite{adem2007orbifolds} for the definition of an effective orbifold using orbifold charts.  
An orbifold chart is denoted by $(V, \Gamma, \phi)$, where:
\begin{enumerate}
  \item $V$ is a connected open subset of $\R^n$;
  \item $\Gamma$ is a finite group acting smoothly and effectively on $V$; and
  \item $\phi: V \to \tp X$ is a $\Gamma$-invariant continuous map such that the induced map $V/\Gamma \to \tp X$ is a homeomorphism onto an open subset $U \subset \tp X$. 
\end{enumerate}

An embedding of orbifold charts is denoted by $$(\psi_{V'V}, \theta_{\Gamma'\Gamma}): (V, \Gamma, \phi) \to (V', \Gamma', \phi'),$$ where:
\begin{enumerate}
  \item $\psi_{V'V}: V \to V'$ is an embedding satisfying $\phi' \circ \psi_{V'V} = \phi$;
  \item $\theta_{\Gamma'\Gamma}: \Gamma \to \Gamma'$ is an injective group homomorphism; and
  \item $\psi_{V'V}$ is $\theta_{\Gamma'\Gamma}$-equivariant.
\end{enumerate}

Let $f: \tp X \to \R$ be a continuous function, and let $(V, \Gamma, \phi)$ be an orbifold chart.  
We write $f_V := f \circ \phi: V \to \R$. 

\begin{definition}
  A function $f: \tp X \to \R$ is called \emph{Morse} if, for every orbifold chart $(V, \Gamma, \phi)$, the local representative $f_V: V \to \R$ is a Morse function.
\end{definition}

By abuse of notation, we will write $f: \X \to \R$ to emphasize that $f$ is defined on the orbifold $\X$.

\begin{theorem}\label{thm: Morse functions are generic}
  Suppose $\tp X$ is compact.  
  For any integer $k \ge 2$, a generic $C^k$-function on $\X$ is Morse.
\end{theorem}

For orbifolds that are global quotients, this statement follows directly from the following classical result.

\begin{theorem}[\cite{wasserman1969equivariant}, Lemma 4.8; see also \cite{bao2024morse}, Theorem~1.2]
  Let $M$ be a closed manifold, and let $G$ be a finite group acting smoothly on $M$.  
  Then the set of $G$-equivariant $C^k$-Morse functions is generic in the space of $G$-equivariant $C^k$-functions.
\end{theorem}

To prove Theorem~\ref{thm: Morse functions are generic} for general orbifolds, we require the following local statement.

\begin{proposition}[\cite{bao2024morse}, Theorem~3.2]\label{prop: morse is generic locally}
  Let $G$ be a finite group, and let $V$ be a $G$-inner product space of dimension $n$, with $k \ge 2$.  
  Then, for any $G$-invariant compact subset $K \subset V$, the set of Morse functions is open and dense in the space of $G$-invariant $C^k$-functions on $K$.
\end{proposition}

\begin{proof}[Outline of the proof of Theorem~\ref{thm: Morse functions are generic}]
  Let $f: \X \to \R$ be a $C^k$-function.  
  It suffices to show that $f$ can be perturbed arbitrarily slightly to become Morse.  
  Cover $\tp X$ by finitely many orbifold charts $\{(V_i, \Gamma_i, \phi_i)\}$ such that the union of closed disks $D_i \subset \phi_i(V_i)$ covers $\tp X$.  
  Suppose inductively that $f$ has already been perturbed to be Morse on $D_1 \cup \dots \cup D_j$. 
  By Proposition~\ref{prop: morse is generic locally}, we can make an arbitrarily small perturbation of $f$ on the chart $(V_{j+1}, \Gamma_{j+1}, \phi_{j+1})$ so that $f$ is Morse on $D_{j+1}$.  
  Since being Morse is an open condition on a compact subset (see \cite{milnor1965lectures}, Lemma~B), it follows that $f$ is Morse on $D_1 \cup \dots \cup D_{j+1}$.
\end{proof}

For each $x \in \tp X$, let $(V, \Gamma, \phi)$ be a coordinate chart containing $x$, i.e., $x \in \phi(V)$, and let $\tilde{x} \in V$ be a lift of $x$.  
Define the stabilizer
\[
\stab(\tilde{x}) := \{\, \mu \in \Gamma \mid \mu \tilde{x} = \tilde{x} \,\} \subset \Gamma.
\]

Let $p \in \tp X$ be a critical point of $f$, and $(V, \Gamma, \phi)$ be an orbifold chart containing $p$.  
We let $\tilde{p} \in V$ be a lift of $p$.  

Denote $T_{\tilde{p}} V^\perp$ to be the kernel of the averaging operator on $T_{\tilde{p}} V$ given by $$\xi \mapsto \sum_{\rho \in \stab(\tilde{p})} \dd \rho \, \xi.$$

\begin{definition}[Stable critical point]
  A critical point $p \in \tp X$ of $f$ is said to be \emph{stable} if there exists an orbifold chart $(V, \Gamma, \phi)$ containing $p$ and a lift $\tilde{p}$ of $p$ such that $\hess f(\tilde{p})$ is positive definite on $T_{\tilde{p}} V^\perp$.
\end{definition}

\begin{remark}
  The definition of stability is independent of the choice of orbifold chart and lift $\tilde{p}$.
\end{remark}

\begin{definition}[Stable Morse function]
  A Morse function $f: \X \to \R$ is said to be \emph{stable} if all of its critical points are stable.
\end{definition}

For each critical point $p$ of $f$, denote by $\ind(p)$ its Morse index.  
Given a stable Morse function $f$, we define the Euler number of $\X$ by
\[
\chi(\X) = \sum_{p \in \crit(f)} (-1)^{\ind(p)} \frac{1}{|\stab(\tilde{p})|},
\]
where $\tilde{p}$ is a lift of $p$ in an orbifold chart containing $p$, and $\ind(p)$ is the Morse index of $p$.  
The quantity $|\stab(\tilde{p})|$ is independent of the choice of chart and lift.  
This definition agrees with Definition~13.3.3 in \cite{thurston2022thegeometry}.

Now we study the genericity of stable Morse functions on an orbifold. We start with the global quotient case. Let $M$ be a closed manifold, and $G$ be a finite group that acts on $M$ smoothly. 

\begin{theorem}[\cite{mayer1989Ginvariant}; \cite{bao2024morse}, Theorem~1.7]\label{thm:stableperturbation}
  Given a $G$-equivariant Morse function $f$ on $M$, there exists a $G$-equivariant \emph{stable} Morse function $f'$ that is arbitrarily close to $f$ in the $C^0$-topology and agrees with $f$ except near the critical points of $f$.
\end{theorem}
The modification is local, so we have the following corollary.
\begin{corollary}\label{corollary: stabilization}
  Given a Morse function $f: \X \to \R$, there exists a \emph{stable} Morse function $f'$ that is arbitrarily close to $f$ in the $C^0$-topology and agrees with $f$ except near the critical points of $f$.
\end{corollary}

We now outline the modification of $f$ near a non-stable critical point $p$, as carried out in the proof of Theorem~1.7 in \cite{bao2024morse}.  
Let $(V, \Gamma, \phi)$ be a coordinate chart containing $p$, and let $\tilde{p} \in V$ be a lift of $p$.  
Let $H = \stab(\tilde{p}) \subset \Gamma$, and then $H$ acts linearly on $T_{\tilde{p}} V$. 
Pick a $\Gamma$-equivariant metric $\mu$ on $V$. With respect to $\mu$, we raise $\hess f(\tilde{p})$ to a self-adjoint $(1,1)$-tensor, and denote its negative eigenspace by $T_{\tilde{p}}^- V$.  
Then
\[
T_{\tilde{p}}^- V = T_{\tilde{p}}^- V^H \oplus T_{\tilde{p}}^- V^\perp,
\]
where $T_{\tilde{p}}^- V^H$ is the subspace of $T^-_{\tilde{p}}V$ fixed by $H$, and $T_{\tilde{p}}^- V^\perp$ is the kernel of the averaging operator on $T^-_{\tilde{p}}V$ that is given by $\xi \mapsto \sum_{\rho \in H} \dd \rho \, \xi$.  
Let $S_p^- \subset T_{\tilde{p}}^- V^\perp$ be the unit sphere. 

Choose a $H$-equivariant stable Morse function $h: S_p^- \to \R$, whose existence is guaranteed by Theorem~1.7 of \cite{bao2024morse}.  
We modify $f$ in a small neighborhood of $p$ along the directions of $T_{\tilde{p}}^- V^\perp$ by adding a quadratic term so that $\tilde{p}$ becomes stable and new stable critical points are introduced corresponding to $\crit(h)$.  
They satisfy
\[
\ind_{f'}(\tilde{p}) = \ind_f(\tilde{p}) - \dim T_{\tilde{p}}^- V^\perp
\]
\[
\ind_{f'}(x) = \ind_f(\tilde{p}) - \dim T_{\tilde{p}}^- V^\perp + \ind_h(x) + 1
\]
where $x \in \crit(h)$ corresponds to a new stable critical point, also denoted as $x$, of the perturbed function $f'$.

\section{The moduli space of Morse flow lines}

We say $g$ is a Riemannian metric on $\X$ if it contains the following data:
\begin{enumerate}
  \item for each orbifold chart $(V, \Gamma, \phi)$, $g_V$ is a $\Gamma$-equivariant Riemannian metric on $V$;
  \item for any orbifold chart embedding $(\psi_{V'V}, \theta_{\Gamma'\Gamma}): (V, \Gamma, \phi) \to (V', \Gamma', \phi')$, we have
  $\psi_{V'V}^* g_{V'} = g_V$.
\end{enumerate}

We denote by $-\grad_f$ the negative gradient vector field of $f$ with respect to the metric $g$.
As in the manifold case, for each critical points $p, q \in \tp X$ of $f$, we can define descending orbifolds $\D{p} \subset X$ and ascending orbifolds $\A{p} \subset X$ in the usual way.

\begin{definition}
  We say the pair $(f,g)$ is \emph{Morse--Smale} if $\D{p}$ is transverse to $\A{q}$ for all critical points $p,q$ of $f$.
\end{definition}

It is known that for certain Morse functions $f$, there does not exist a metric $g$ such that the pair $(f,g)$ is Morse--Smale. However, if the Morse function is stable, then there is no obstruction.

\begin{theorem}[Orbifold Smale theorem]\label{thm: morse smale is generic fixing a stable f}
  Let $k \geq 1$ be an integer, and $\X$ be a compact effective orbifold. Let $f$ be a stable Morse function on $\X$ of class $C^{k+1}$. Then, for a generic metric $g$ on $\X$ of class $C^k$, the pair $(f, g)$ is Morse-Smale.
\end{theorem}

This theorem follows directly from the proof of Theorem~1.8 in \cite{bao2024morse}, which treats the global quotient orbifold case.

For critical points $p, q$ of $f$, we define the moduli space $\M{p}{q} = (\D{p} \cap \A{q})/\R$, where $\R$ is the direction of the flow.
Alternatively, we can describe $\M{p}{q}$ as follows.  
We say a map $\gamma: \R \to \tp X$ is a Morse flow line from $p$ to $q$ if, over each chart and for each local lift $\tilde{\gamma}$ of $\gamma$, we have
\[
\frac{\dd \tilde{\gamma}}{\dd s} = -\grad_f(\tilde{\gamma}),
\]
and $\gamma(-\infty) = p$ and $\gamma(+\infty) = q$.  
Equivalently, we can define $\M{p}{q}$ to be the set of equivalence classes of Morse flow lines from $p$ to $q$, where two flow lines are equivalent if they differ by a translation in the domain $\R$.
We denote by $\overline{\mathcal M}(p,q)$ the compactified moduli space obtained by adding broken Morse flow lines from $p$ to $q$.

\begin{lemma}\label{lemma: stabilizers}
  For any $s, s' \in \R$, let $(V,\Gamma, \phi)$ and $(V',\Gamma', \phi')$ be orbifold charts containing $\gamma(s)$ and $\gamma(s')$, respectively.  
  A lift $\tilde{\gamma}(s) \in V$ canonically determines a lift $\tilde{\gamma}(s') \in V'$, and $\stab(\tilde{\gamma}(s)) \subset \Gamma$ is canonically isomorphic to $\stab(\tilde{\gamma}(s')) \subset \Gamma'$.
\end{lemma}

\begin{proof}
Suppose $\gamma(s)$ and $\gamma(s')$ both lie in the same orbifold chart $(V, \Gamma, \phi)$.
Let $\tilde{\gamma}_V$ be a continuous lift of $\gamma$ over $V$, and let $\varphi^t_V$ denote the Morse flow on $V$.  
Then $\dd \varphi^{s'-s}_V$ takes $T_{\tilde{\gamma}(s)}V$ isomorphically onto $T_{\tilde{\gamma}(s')}V$.  
Since the Morse function $f$ is $\Gamma$-equivariant, the map $\dd \varphi^{s'-s}_V$ is also $\Gamma$-equivariant.  
Hence $\stab(\tilde{\gamma}(s)) = \stab(\tilde{\gamma}(s'))$.

Now suppose $\gamma(s)$ and $\gamma(s')$ lie in orbifold charts $(V, \Gamma, \phi)$ and $(V', \Gamma', \phi')$, respectively, and that $(V', \Gamma', \phi')$ embeds into $(V, \Gamma, \phi)$ via an embedding $\psi_{VV'}: V' \to V$ and an injective homomorphism $\theta_{\Gamma\Gamma'}: \Gamma' \to \Gamma$.  
Since $\Gamma$ acts effectively on $V$, the stabilizer $\stab(\tilde{\gamma}_V(s')) \subset \Gamma$ acts effectively on a neighborhood of $\tilde{\gamma}_V(s')$ in $V$.  
This shows that $\theta_{\Gamma\Gamma'}$ maps $\stab(\tilde{\gamma}(s'))$ isomorphically onto $\stab(\tilde{\gamma}(s))$ (See Section 1.1 of \cite{adem2007orbifolds}).

In general, we cover the image of $\gamma$ by finitely many sufficiently small charts, and the statement follows inductively from the above argument.
\end{proof}

Suppose $f$ is stable.
For each $[\gamma] \in \M{p}{q}$, let $(V,\Gamma, \phi)$ be an orbifold chart containing $p$, and let $\tilde{p}$ be a lift of $p$.
Then $\stab(\tilde{p}) \subset \Gamma$ fixes $T_{\tilde{p}}^- V$.  
This implies that for each $s \in \R$, $\stab(\tilde{\gamma}(s))$ is canonically isomorphic to $\stab(\tilde{p})$.  
In other words, $\stab(\tilde{\gamma}(s))$ depends only on $\tilde{p}$ and not on $s$ or $[\gamma]$.  
Taking $s \to \infty$, we obtain that $\stab(\tilde{p})$ is canonically isomorphic to a subgroup of $\stab(\tilde{q})$. This implies the following lemma.

\begin{lemma}\label{lemma: quotient is an integer}
  Under the above assumption, $|\stab(\tilde{p})|$ divides $|\stab(\tilde{q})|$.
\end{lemma}

To orient moduli spaces of Morse flow lines, we need orientation data: choices of orientations of the descending manifolds for each critical point.

\begin{theorem}\label{prop: moduli space is a manifold}
  Given a stable Morse function $f$ and a metric $g$ such that the pair $(f,g)$ is Morse--Smale,  
  there exist consistent choices of orientation data such that  
  for any critical points $p,q$ of $f$, the moduli space $\M{p}{q}$ is a smooth, oriented manifold of dimension $\ind p - \ind q - 1$.  
  In the case $\ind p - \ind q = 2$, the compactified moduli space $\overline{\mathcal M}(p,q)$ is a one-dimensional manifold whose boundary $\p\overline{\mathcal M}(p,q)$, as an oriented manifold, is diffeomorphic to  
  \[
  \coprod_{\substack{r \in \crit(f)\\ \ind(r) = \ind p - 1}} \M{p}{r} \times \M{r}{q}.
  \]
\end{theorem}

\begin{remark}
  Without the stability assumption, in general, transversality only implies that the moduli spaces are orbifolds.
\end{remark}

\begin{example}
  Consider the global quotient orbifold $\X = [S^2/C_2]$ obtained by taking $S^2 \subset \R^3$ modulo $C_2 = \langle \sigma \rangle$, where $\sigma(x, y, z) = (-x, y, z)$.  
  Let $f: \X \to \R$ be projection onto the $z$-component.  
  Let $p = (0,0,1)$ and $q = (0,0,-1)$ be the critical points of $f$.  
  Then with respect to the Euclidean metric, the moduli space $\M{p}{q}$ is diffeomorphic to $[0,1]$, which is not a manifold.
\end{example}

\begin{proof}[Proof of Theorem~\ref{prop: moduli space is a manifold}]
  For each $[\gamma] \in \M{p}{q}$, we can cover the image of $\gamma$ with orbifold charts and lift $\gamma$ to $\tilde{\gamma}$ over each chart.  
  This also lifts nearby elements $[\gamma'] \in \M{p}{q}$ around $[\gamma]$.  
  The Morse--Smale condition implies that the lifted neighborhood of $[\gamma]$ is a smooth manifold.  
  Since each element $[\tilde{\gamma}']$ has the same stabilizer $\stab(\tilde{p})$, the moduli space $\M{p}{q}$ is a manifold.

  The orientation argument is the same as in the manifold case: we use Equation~\eqref{eqn: orientation} to determine the orientation of the moduli spaces.
\end{proof}

\section{Invariant and Coinvariant Chain Complexes}

Let $f$ be a stable Morse function from here on out. We define invariant and coinvariant chain complexes.

\subsection{Orientation}
To orient the moduli spaces of Morse flow lines, for each critical point $p \in \crit(f)$ we choose an orientation $o_{\tilde{p}}$ of the negative eigenspace $T_{\tilde{p}}^-V$, where $(V, \Gamma, \phi)$ is an orbifold chart around $p$, and $\tilde{p} \in V$ is a lift of $p$.  
The lifts of $p$ to $V$ are given by $\mu \tilde{p}$ for $\mu \in \Gamma$.  
We define the orientation $o_{\mu \tilde{p}}$ by
\begin{equation}
  o_{\mu \tilde{p}} = \dd\mu|_{\tilde{p}}\, o_{\tilde{p}}.
\end{equation}
This definition is independent of the choice of $\mu$.  
Indeed, if $\mu \tilde{p} = \mu' \tilde{p}$, then $\dd\mu|_{\tilde{p}}\, o_{\tilde{p}} = \dd\mu'|_{\tilde{p}}\, o_{\tilde{p}}$, since $\mu^{-1}\mu' \in \stab(\tilde{p})$ acts trivially on $T^-_{\tilde{p}}V$ by the stability of $p$.

\begin{remark}
  In the unstable case, $\stab(\tilde{p})$ may not preserve the orientation of $T_{\tilde{p}}^- V$.  
  In \cite{cho2014orbifold}, such critical points are excluded from the chain complex.
\end{remark}

For any $[\gamma] \in \M{p}{q}$, 
choose a continuous lift $\tilde{\gamma}$ of $\gamma$.  
Let $(V, \Gamma, \phi)$ and $(V', \Gamma', \phi')$ be coordinate charts containing $p$ and $q$, respectively.  
We define the orientation $o_{[\gamma]}$ of $T_{[\gamma]} \M{p}{q}$ so that the following isomorphism, given by the linearization of the Morse flow along $\gamma$, is orientation-preserving:
\begin{equation}\label{eqn: orientation}
  T_{\tilde{p}}^- V \simeq T_{[\gamma]} \M{p}{q} \oplus \R\langle \p_s \rangle \oplus T^-_{\tilde{q}}V',
\end{equation}
where $s$ denotes the coordinate of $\R$.  
When $\ind(p) - \ind(q) = 1$, we have $o_{[\gamma]} \in \{1, -1\}$.

\subsection{Coinvariant Chain Complex}
We now define the coinvariant chain complex.  
Let $\coC = \Z\langle \crit(f) \rangle$, graded by the Morse index.  
For each critical point $p\in \crit(f)$, denote the corresponding generator in $\coC$ by $[p]$.
Define the differential $\cop: \coC \to \coC$ by
\[
  \cop[p] = \sum_{\substack{q \in \crit(f) \\ \ind(q) = \ind(p) - 1}} 
  \sum_{[\gamma] \in \M{p}{q}} \op{o}_{[\gamma]}\cdot [q].
\]

\begin{lemma}
  $(\cop)^2 = 0$.
\end{lemma}
\begin{proof}
  This follows from Theorem~\ref{prop: moduli space is a manifold}.
\end{proof}

The coinvariant homology is defined as $\coH(\X) = \ker \cop / \op{im} \cop$.

\begin{theorem}
  $\coH(\X)$ is an invariant of the orbifold $\X$, i.e., it is independent of the choice of stable Morse function and metric.
\end{theorem}
The proof is given in Section~\ref{section: proof of invariance}.

\begin{conjecture}
  $\coH(\X)$ is isomorphic to the homology of $\tp X$ over $\Z$.
\end{conjecture}

This generalizes Theorem~2.9 of \cite{cho2014orbifold}, which is over $\Q$.

\subsection{Invariant Chain Complex}
We next consider the invariant chain complex.  
Let $\inC = \Z\langle \crit(f) \rangle$, graded by the Morse index.  
For each critical point $p \in \crit(f)$, denote the corresponding generator in $\inC$ by $\bar{p}$.  
Define the differential $\inp: \inC \to \inC$ by
\[
  \inp\bar{p} = 
  \sum_{\substack{q \in \crit(f) \\ \ind(q) = \ind(p) - 1}} 
  \sum_{[\gamma] \in \M{p}{q}} 
  \op{o}_{[\gamma]} \frac{|\stab(q)|}{|\stab(p)|}\, \bar{q},
\]
where $\stab(p)$ and $\stab(q)$ denote the stabilizer groups of $p$ and $q$, respectively. 

\begin{lemma}
  $(\inp)^2 = 0$.
\end{lemma}
\begin{proof}
  Let $p$ and $r$ be critical points of $f$ with $\ind(p) - \ind(r) = 2$. Then we compute the coefficent of $\bar{r}$ in $(\inp)^2\bar{p}$:
  \begin{align*}
    \langle (\inp)^2\bar{p}, \bar{r} \rangle 
    & = \frac{|\stab(r)|}{|\stab(p)|} 
      \sum_{\substack{q \in \crit(f) \\ \ind(q) = \ind(p) - 1}} 
      \sum_{[\gamma]\in \M{p}{q}} 
      \sum_{[\eta] \in \M{q}{r}} 
      o_{[\gamma]} o_{[\eta]} \\
    & = \frac{|\stab(r)|}{|\stab(p)|} \langle (\cop)^2[p], [r] \rangle \\
    & = 0.
  \end{align*}
\end{proof}

We define the invariant homology as $\inH = \ker \inp / \op{im} \inp$.

\begin{theorem}
  $\inH(\X)$ is an invariant of the orbifold $\X$, i.e., it is independent of the choice of stable Morse function and metric.
\end{theorem}
The proof is given in Section~\ref{section: proof of invariance}.

\begin{example} \label{example: sphere}
  Consider the orbifold $\X$ whose underlying topological space is $\tp X = S^2$.  
  It has two orbifold points $p$ and $q$ with stabilizers $\Z_m$ and $\Z_n$, respectively.  
  Suppose we choose a Morse function $f$ with an index-$2$ critical point at $p$ and an index-$0$ critical point at $q$.  
  Then $f$ is unstable at $p$. We stabilize it by introducing an index-$2$ critical point $p''$ and an index-$1$ critical point $p'$ near $p$, as illustrated in Figure~\ref{fig:perturbed}. This corresponds to equivariantly adding $m$ index-$2$ critical points and $m$ index-$1$ critical points in $S^2$ near $p$ as illustrated in Figure~\ref{fig:perturbed}.

  The orbifold Euler characteristic is $\chi(\X) = 1/m + 1/n$.

  For the coinvariant chain complex, we have 
  $\cop [p''] = 0$ since there are two flow lines from $p''$ to $p'$ with opposite orientations as shown in Figure~\ref{fig:perturbed};  
  $\cop [p'] = [p] - [q]$; and $\cop [p] = \cop[q] =0$.  
  This yields the homology 
  \[\coH_2 (\X) = \Z, \quad  \quad  \coH_1(\X) = 0,  \quad  \quad \coH_0(\X)= \Z,\] 
  which coincides with the homology of $\tp X$.

  For the invariant chain complex, we have 
  $\inp \bar{p}'' = 0$, 
  $\inp \bar{p}' = m \bar{p} - n \bar{p}'$, 
  and $\inp \bar{p} = \inp \bar{q}= 0 $.  
  This gives 
  \[\inH_2 (\X) = \Z, \quad  \quad \inH_1(\X) = 0,  \quad  \quad \inH_0(\X)= \Z \oplus \Z_\ell,\] 
  where $\ell = \op{gcd}(m, n)$.  
  Note that $\X$ is a global quotient if and only if $\ell = 1$.
\end{example}

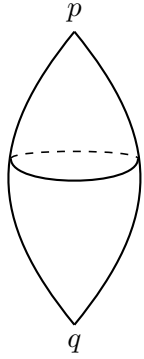
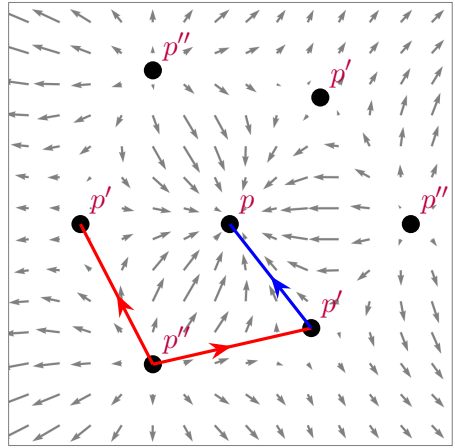
\begin{figure}[h]
\centering

\begin{subfigure}[t]{0.35\textwidth}
  \centering
  \begin{tikzpicture}[scale=1.3, line join=round]
    \draw[line width=0.8pt]
      (0,1.5) .. controls (-0.9,0.4) and (-0.9,-0.4) .. (0,-1.5);
    \draw[line width=0.8pt]
      (0,1.5) .. controls (0.9,0.4) and (0.9,-0.4) .. (0,-1.5);

    \draw[line width=0.8pt]
      (-0.65,0.2) .. controls (-0.65,-0.1) and (0.65,-0.1) .. (0.65,0.2);
    \draw[dashed,line width=0.6pt]
      (-0.65,0.2) .. controls (-0.65,0.3) and (0.65,0.3) .. (0.65,0.2);

    \node[above] at (0,1.5) {$p$};
    \node[below] at (0,-1.5) {$q$};
  \end{tikzpicture}
  \caption{$S^2$ with orbifold points $p$ and $q$, which are also the critical points of the height function.}
  \label{fig: sphere with two critical points}
\end{subfigure}
\hspace{1cm}
\begin{subfigure}[t]{0.55\textwidth}
  \centering
  \begin{tikzpicture}[scale=1.3, line join=round, >=Stealth]
    \node[anchor=south west, inner sep=0] (img) at (0,0)
        {\includegraphics[width=6cm]{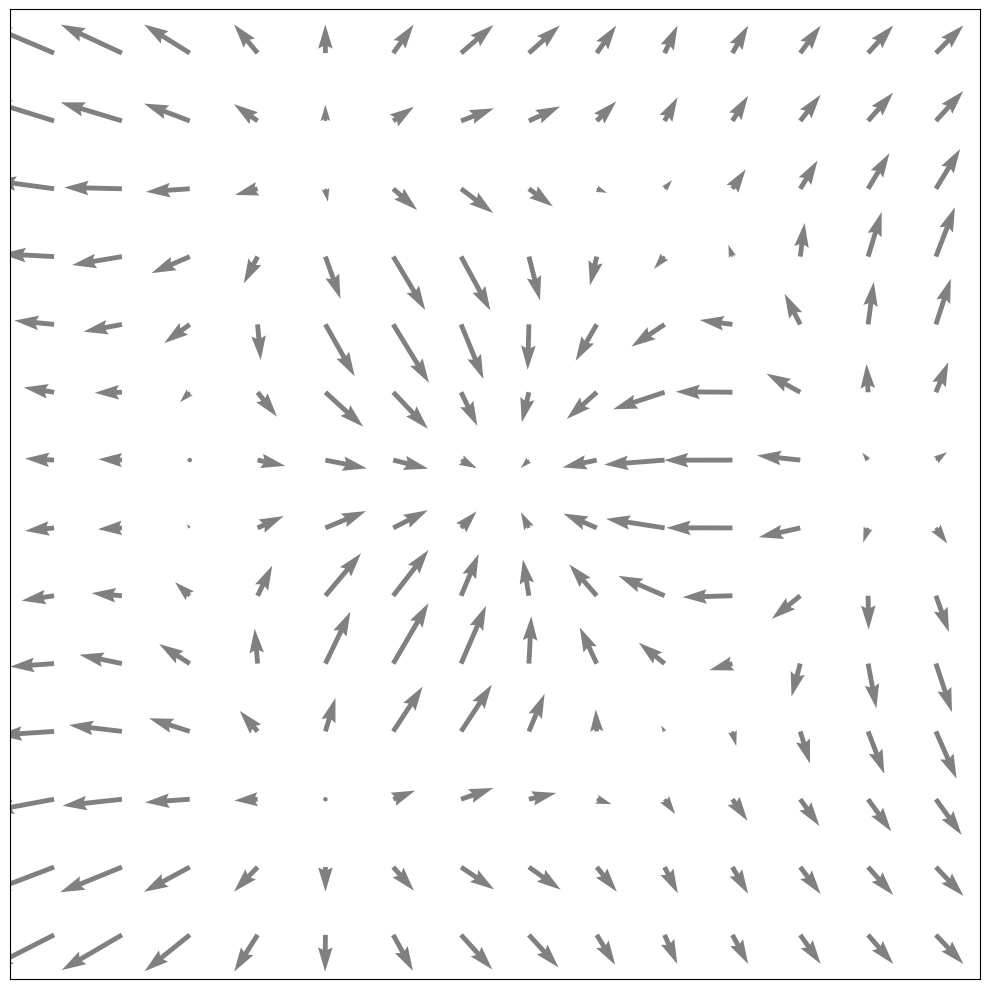}};
    \begin{scope}[x={(img.south east)}, y={(img.north west)}]

      \coordinate (Pcoord)  at (0.50,0.50);
      \coordinate (P1coord) at (0.90,0.50);
      \coordinate (P2coord) at (0.33,0.84);
      \coordinate (P3coord) at (0.33,0.19);
      \coordinate (P4coord) at (0.70,0.78);
      \coordinate (P5coord) at (0.17,0.5);
      \coordinate (P6coord) at (0.68,0.27);

      \foreach \p in {Pcoord,P1coord,P2coord,P3coord,P4coord,P5coord,P6coord}
        \fill (\p) circle (0.02);

      \node[above right] at (Pcoord)  {$\color{purple} p$};
      \node[above right] at (P1coord) {$\color{purple} p''$};
      \node[above right] at (P2coord) {$\color{purple} p''$};
      \node[above right] at (P3coord) {$\color{purple} p''$};
      \node[above right] at (P4coord) {$\color{purple} p'$};
      \node[above right] at (P5coord) {$\color{purple} p'$};
      \node[above right] at (P6coord) {$\color{purple} p'$};

      \draw[red, very thick, postaction={decorate, decoration={markings, mark=at position 0.5 with {\arrow{>}}}}]
        (P3coord) -- (P5coord);
      \draw[red, very thick, postaction={decorate, decoration={markings, mark=at position 0.5 with {\arrow{>}}}}]
        (P3coord) -- (P6coord);
      \draw[blue, very thick, postaction={decorate, decoration={markings, mark=at position 0.5 with {\arrow{>}}}}]
        (P6coord) -- (Pcoord);

    \end{scope}
  \end{tikzpicture}
  \caption{Negative gradient vector field of the perturbed Morse function in the orbifold chart centered at $p$ (case $m = 3$).}
  \label{fig:local perturbed}
\end{subfigure}

\caption{Stabilization of the critical point $p$.}
\label{fig:sphere_perturbed}
\end{figure}

In the above example, when $m = n = 2$, the orbifold $\X$ is a global quotient. Now we compute the invariant and coinvariant homologies with a different Morse function.
\begin{example}
  Consider the global quotient orbifold $\X = [S^2 / \langle \sigma \rangle]$, where $\sigma: S^2 \to S^2$ is the rotation about the $z$-axis by an angle of $\pi$.  
  Embed $S^2$ as a bean-shaped surface, as illustrated in Figure~\ref{fig: bean under rotation}, and let the Morse function be the height function.  
  In this case, there are three critical points $p$, $q$, and $r$. The group action reverses the orientation of the descending manifold $\mathcal D_q$.  
  In \cite{cho2014orbifold}, the critical point $q$ is omitted from the chain complex.  
  Since $q$ is unstable, we stabilize it by adding an index-$1$ critical point $q'$ near $q$, making the index of $q$ equal to $0$ with respect to the new morse function.  

  For the coinvariant chain complex, we have $\cop [p] = 0$ since the two flow lines from $p$ to $q'$ cancel; $\cop [q'] = [q] - [r]$; and $\cop [r] = \cop [q] = 0$.  
  This yields the homology \[\coH_2 = \Z, \quad  \quad  \coH_1 = 0,  \quad  \quad  \coH_0 = \Z.\]  

  For the invariant chain complex, we have $\inp \bar{p} = 0$, $\inp \bar{q}' = 2\bar{q} - 2\bar{r}$, and $\inp \bar{r} = \inp \bar{q} = 0$.  
  This yields the homology \[\inH_2 = \Z, \quad  \quad  \inH_1 = 0, \quad  \quad  \inH_0 = \Z \oplus \Z_2.\]
\end{example}

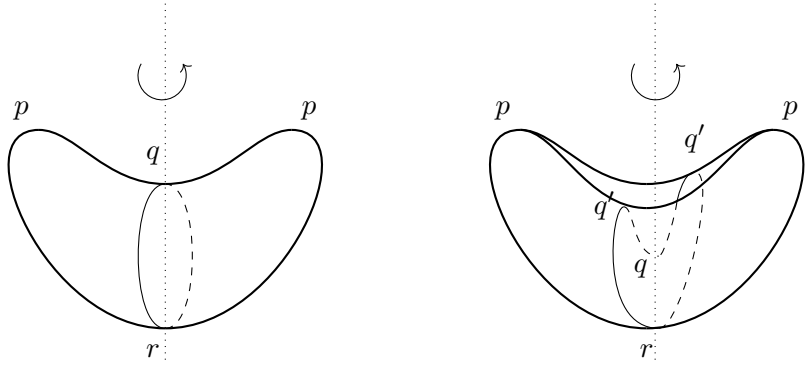
\begin{figure}[h]
\centering

\begin{subfigure}[t]{0.3\textwidth}
  \centering
  \begin{tikzpicture}[scale=0.8, line join=round, >=Stealth]
    \node[anchor=south west, inner sep=0] (img) at (0,0)
        {\includegraphics[width=5cm]{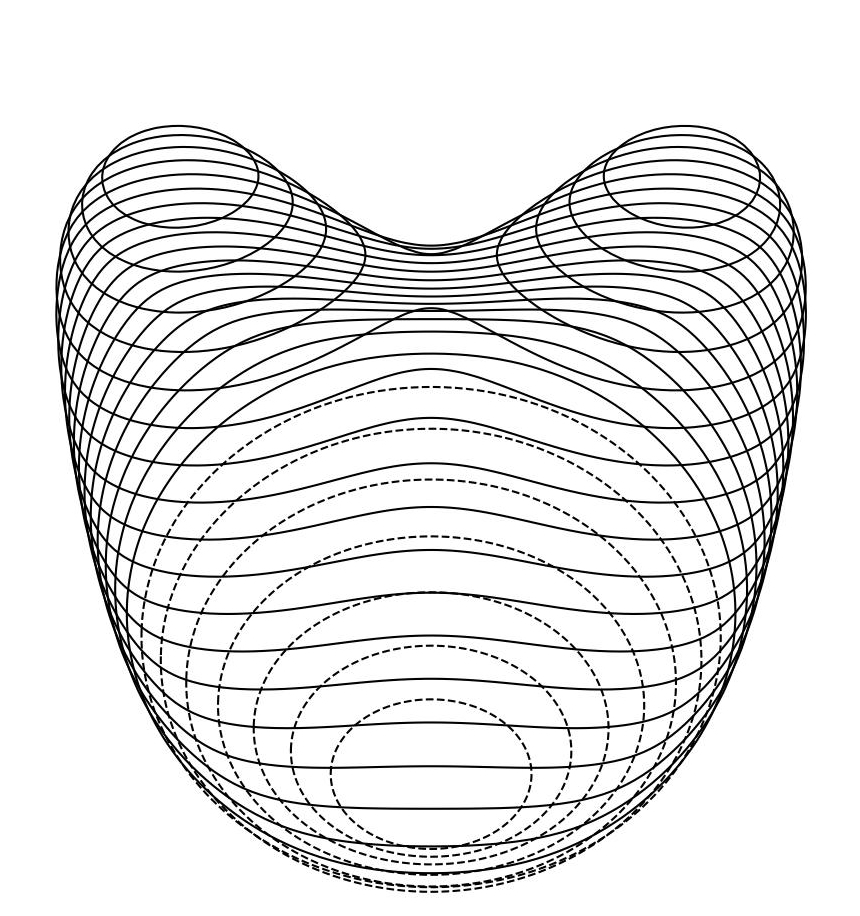}};
    \begin{scope}[x={(img.south east)}, y={(img.north west)}]
    
      \draw[->] 
    (0.46, 1) arc[start angle=150, end angle=390, radius=0.05];

      \fill[black] (0.2,0.83) circle (0.01);
      \node[above] at (0.2,0.83) {$\color{black} p$};

      \fill[black] (0.8,0.83) circle (0.01);
      \node[above] at (0.8,0.83) {$\color{black} p$};

      \fill[blue] (0.5,0.7) circle (0.01);
      \draw[dotted] (0.5,0) -- (0.5,1);


    \fill[black] (0.5,0.15) circle (0.01);
    \node[above] at (0.48,0.15) {$\color{black} r$};

    \end{scope}
  \end{tikzpicture}
  \caption{The center blue dot is the critical point $q$ before stabilization.}
  \label{fig:perturbed}
\end{subfigure}
\hspace{2cm}
\begin{subfigure}[t]{0.3\textwidth}
  \centering
  \begin{tikzpicture}[scale=0.8, line join=round, >=Stealth]
    \node[anchor=south west, inner sep=0] (img) at (0,0)
        {\includegraphics[width=5.3cm]{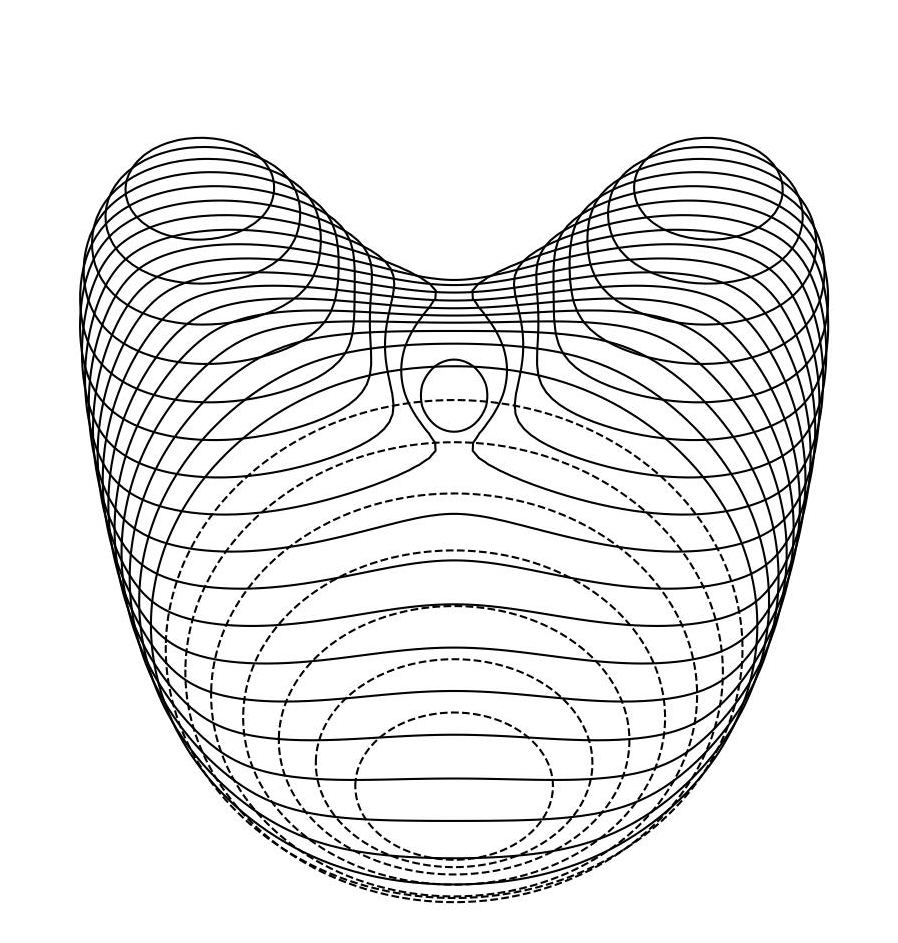}};
    \begin{scope}[x={(img.south east)}, y={(img.north west)}]

      \draw[->] 
    (0.46, 1) arc[start angle=150, end angle=390, radius=0.05];
      
          \fill[black] (0.2,0.83) circle (0.01);
      \node[above] at (0.2,0.83) {$\color{black} p$};

      \fill[black] (0.8,0.83) circle (0.01);
      \node[above] at (0.8,0.83) {$\color{black} p$};

      \fill[red] (0.5,0.68) circle (0.01);

      \fill[red] (0.5,0.51) circle (0.01);
      
      \fill[blue] (0.5,0.57) circle (0.01);

      \fill[black] (0.5,0.15) circle (0.01);
    \node[above] at (0.48,0.15) {$\color{black} r$};

      \draw[dotted] (0.5,0) -- (0.5,1);
    \end{scope}
  \end{tikzpicture}
  \caption{The red dots (identified by rotation) are new critical point $q'$ after stabilization.}
  \label{fig:perturbed}
\end{subfigure}

\caption{The bean shaped $S^2$ quotient by $\pi$ rotation.}
\label{fig: bean under rotation}
\end{figure}

\section{Proof of Invariance}\label{section: proof of invariance}
To show that the homologies $\coH(\X)$ and $\inH(\X)$ are independent of the choice of stable Morse function and metric, we follow the standard argument in Morse theory using continuation maps. The proof of Theorem~\ref{thm: morse smale is generic fixing a stable f} in \cite{bao2024morse} requires a perturbation of the metric around the critical points.
To make the proof fit into the standard framework, we give an alternative proof of Theorem~\ref{thm: morse smale is generic fixing a stable f}. 
In terms of transversality, there is no loss of generality in working in the global quotient case. 

Let $M$ be a closed manifold and $G$ be a finite group acting on $M$ effectively and smoothly.
Let $f$ be a $G$-invariant stable Morse function on $M$. For any critical points $p$ and $q$ of $f$, define $$\mathcal B = \mathcal B(p,q) = \{ \gamma \in W^{1,2}(\R, M) : \lim_{s \to -\infty} \gamma(s) = p, \lim_{s \to +\infty} \gamma(s) = q \}.$$
Let $\mathcal J = \mathcal J^k$ be the space of $G$-invariant Riemannian metrics on $M$ of class $C^k$ for some $k \geq 2$.
Consider the bundle $\mathcal E \to \mathcal J \times \mathcal B$ whose fiber over $(g, \gamma)$ is $L^2(\R, \gamma^* TM)$.
Denote by $\mathbb L: \mathcal J \times \mathcal B \to \mathcal E$ the section defined by $$\mathbb L(g, \gamma) = \frac{\dd \gamma}{\dd s} + \op{grad}_{f,g}(\gamma).$$

\begin{proposition}
  The section $\mathbb L$ is transverse to the zero section of $\mathcal E$.
\end{proposition}
This is the standard transversality result when $G$ is trivial. For example see \cite[Proposition 5.8]{hutchings2002lecture}.
\begin{proof}
  For any $(g, \gamma) \in \mathbb L^{-1}(0)$, denote by $\mathbb D_{(g, \gamma)}: T_g \mathcal J \times T_\gamma \mathcal B \to \mathcal E_{(g, \gamma)}$ the linearization of $\mathbb L$ at $(g, \gamma)$. 
  In particular, for any $h \in T_g \mathcal J$ and $\xi \in T_\gamma \mathcal B$, we have
  \[
  \mathbb D_{(g, \gamma)}(h, \xi) = \nabla_s \xi + H\xi + Fh,
  \] 
  where $\nabla$ is the Levi-Civita connection of $g$, and $H \xi = \nabla_\xi \op{grad}_{f,g}(\gamma)$, and $Fh$ is the variation of $\op{grad}_{f,g}(\gamma)$ with respect to $h$.
  Let $\varsigma \in \mathcal E_{(g, \gamma)}$ be an element in the $L^2$-orthogonal complement of the image of $\mathbb D_{(g, \gamma)}$, i.e.,
  \[
  \langle \mathbb D_{(g, \gamma)}(h, \xi), \varsigma \rangle = 0 \quad \text{for all } (h, \xi) \in T_g \mathcal J \times T_\gamma \mathcal B.
  \]
  We want to show that $\varsigma \equiv 0$.
  Then $\varsigma \in \mathcal E_{(g, \gamma)}$ satisfies the following equations:
  \begin{align*}
    \langle \nabla_s \xi + H\xi, \varsigma \rangle & = 0 \quad \text{for all } \xi \in T_\gamma \mathcal B, \\
    \langle Fh, \varsigma \rangle & = 0 \quad \text{for all } h \in T_g \mathcal J.
  \end{align*}
 Since $H$ is self-adjoint, the first equation implies 
  \[
  -\nabla_s \varsigma + H \varsigma = 0.
  \]
Since $\varsigma$ satisfies the above ordinary differential equation, it is nowhere vanishing unless $\varsigma \equiv 0$. For any $s \in \R$, let $x = \gamma(s)$. Then by the proof of Lemma~\ref{lemma: stabilizers}, we have $\stab(x) = \stab(p) =: K$. Then $\varsigma(s) = \varsigma(s)^K + \varsigma(s)^\perp$, where $\varsigma(s)^K$ is the $K$-invariant part of $\varsigma(s)$ and $\varsigma(s)^\perp$ is the complement of $\varsigma(s)^K$ in $T_x M$ defined as the kernel of the averaging operator $$\frac{1}{|K|} \sum_{\varphi \in K} \dd \varphi|_x: T_x M \to T_x M.$$ 
Note that the averaging operator acts as the identity on the $K$-invariant part.
Since $H$ is $G$-equivariant, it is also $K$-equivariant. Hence, we have 
\[
-\nabla_s \varsigma^K + H \varsigma^K = 0, \quad -\nabla_s \varsigma^\perp + H \varsigma^\perp = 0.
\]
Suppose $\varsigma(s)^K \neq 0$ for some $s$.
By \cite[Proof of Proposition 5.8]{hutchings2002lecture}, we can choose a variation $h'$ of $g$ at $x$ such that $Fh'= \varsigma^K(s)$. By averaging $h'$ over $K$, we can further assume that $h'$ is $K$-invariant. Then we take a variation $h$ of $g$ such that $h$ coincides with $h'$ at $x$ and is supported in a small neighborhood of $x$. This gives 
\[
\langle Fh, \varsigma^K \rangle > 0,
\]
which is a contradiction to the second equation above. Hence, we have $\varsigma^K \equiv 0$. 

Denote by $H_q: T_q M \to T_q M$ the Hessian of $f$ at $q$ raised to a $(1,1)$-tensor using the metric $g$.
Denote by $$\dots \leq \lambda_{-2} \leq \lambda_{-1} < 0 < \lambda_1 \leq \lambda_2 \leq \cdots$$ the eigenvalues of $H_q$ ($0$ is not an eigenvalue by the non-degeneracy of the critical point $q$), and by $v_i$ the eigenvectors of $H_q$ with eigenvalues $\lambda_i$.

Suppose $\varsigma^\perp$ is not identically zero. For $s$ sufficiently close to $+\infty$, $\varsigma^\perp$ has a leading term $a_i(s) e^{\lambda_i s} v_i$ for some $i < 0$ where $\lim_{s \to \infty} a_i(s)$ exists and is nonzero. 
Since $q$ is stable, we have $v_i \in T_q M$ is invariant under the action of $\stab(q)$. By the proof of Lemma~\ref{lemma: stabilizers}, we have $\stab(q) \subseteq \stab(p) = K$.
Hence $v_i$ is also invariant under the action of $K$. This contradicts to the fact that $\varsigma(s)^\perp$ is in the kernel of the averaging operator $\frac{1}{|K|} \sum_{\varphi \in K} \dd \varphi|_x: T_x M \to T_x M$.
\end{proof}

The rest of the proof of Theorem~\ref{thm: morse smale is generic fixing a stable f} is standard and can be found in \cite{salamonMorse, hutchings2002lecture}. For completeness, we outline the argument here. Let $\mathcal X = \mathbb L^{-1}(0)$ be the zero set of $\mathbb L$. Then $\mathcal X$ is a smooth Banach manifold since $\mathbb L$ is transverse to the zero section. Consider the projection $\pi: \mathcal X \to \mathcal J$. By the Sard--Smale theorem, the set of regular values of $\pi$ is residual in $\mathcal J$. For any regular value $g$ of $\pi$, one can show that the operator $\mathbb L(g, \cdot): \mathcal B \to \mathcal E$ is transverse to the zero section, which implies that $\mathbb L(g, \cdot)^{-1}(0)$—the space of parametrized flow lines from $p$ to $q$—is a smooth manifold of dimension $\ind(p) - \ind(q)$. This completes the proof of Theorem~\ref{thm: morse smale is generic fixing a stable f}.

To show that the homology $\coH(\X)$ is independent of the choice of stable Morse function and metric, we construct a continuation map $$\Phi_{10}: \coC(f_0, g_0) \to \coC(f_1, g_1)$$ for any two stable Morse--Smale pairs $(f_0, g_0)$ and $(f_1, g_1)$ by choosing a generic path $\{(f_t, g_t)\}_t$. We also construct a continuation map $$\Phi_{01}: \coC(f_1, g_1) \to \coC(f_0, g_0)$$ using the reversed path $\{(f_{1-t}, g_{1-t})\}_t$. Then we show that $\Phi_{10} \circ \Phi_{01}$ is chain homotopic to the identity on $\coC(f_1, g_1)$ and $\Phi_{01} \circ \Phi_{10}$ is chain homotopic to the identity on $\coC(f_0, g_0)$ by choosing generic homotopies of paths. The details can be found in \cite[Section 4]{hutchings2002lecture}. This completes the proof of the invariance of $\coH(\X)$. 

The invariance of $\inH(\X)$ follows in the same way.

\bibliographystyle{alpha}
\bibliography{mybib}

\end{document}